\documentclass[final,3p,times]{elsarticle}

\usepackage{amssymb}
\usepackage{graphics,graphicx,amsmath,latexsym,amssymb,amsthm,amssymb,amsfonts}
\newtheorem{thm}{Theorem}[section]
\newtheorem{lem}[thm]{Lemma}

\newtheorem{defn}[thm]{Definition}
\newtheorem{rem}[thm]{Remark}
\newtheorem{prop}[thm]{Proposition}
\newtheorem{exm}[thm]{Example}

\biboptions{square}

\journal{XXXXXXX}

\begin{document}
\date{~}
\begin{frontmatter}

\author{J. Mahanta 
}
 \ead{jm$\_$nerist@yahoo.in}
 \author{P. K. Das 
}
\ead{pkd$\_$ma@yahoo.com}
\address{Department of Mathematics, NERIST, Nirjuli, Arunachal Pradesh, 791 109, INDIA.\\ }

\title{On fuzzy weakly-closed sets}

\begin{abstract} 

A new class of fuzzy closed sets, namely fuzzy weakly closed set in a fuzzy topological space is introduced and it is established that this class of fuzzy closed sets lies between fuzzy closed sets and fuzzy generalized closed sets. Alongwith the study of fundamental results of such closed sets, we define and characterize fuzzy weakly compact space and fuzzy weakly closed space.
\end{abstract}

\end{frontmatter}

\section{Introduction and Preliminaries}
 
It is known that the different notions of topology are defined in terms of the open sets, which can also be replaced by either closed sets, or closure or interior, or neighborhood system. Among these again open sets and closed sets are very closely linked and are indispensable in the study of the topological notions. This has led to several generalizations of closed sets (\cite{SPA}, \cite{SSB}, \cite{BY}, \cite{NL}, \cite{FA}, \cite{HM}, \cite{AS}, \cite{PA} etc.) which in turn have provided several new properties of topological spaces.\\

 In this study, we consider fuzzy topological space (FTS) in Chang's sense \cite{CA}, by $xcl(A)$ we mean $x-$closure of $A$ where $x \in \{\alpha, s, p, sp, g^*s\}$ . We give below a few definitions and results which are required for our study and refer the readers to \cite{PA1} for the terms not defined here.

\begin{defn}
A family $\{A_\lambda~|~ \lambda \in \Lambda\}$ of fuzzy subsets of a set $X$ has the finite intersection property (FIP) if the intersection of the members of each finite subfamily of $\{A_\lambda~|~ \lambda \in \Lambda\}$ is nonempty.
\end{defn}

\begin{defn} \cite{SG}
A collection of fuzzy subsets $\Gamma$ of an FTS $X$ is said to form a fuzzy filterbase iff for every collection $\{A_j~|~ j = 1,2, . . . , n\}, ~ \overset{n}{\underset{j=1}{\bigwedge}} A_j\neq 0_X$.
\end{defn}

\begin{defn} 
A fuzzy set $A$ of an FTS $(X, \tau)$ is said to be:\\
\begin{enumerate}
\item \cite{w4} fuzzy $\alpha-$closed (simply F$\alpha-$closed) if $cl(Int(cl(A))) \leq A$.
\item \cite{w4} fuzzy pre-closed set (simply Fp-closed) if  $cl(Int(A)) \leq A$. 
\item \cite{M} fuzzy semi-closed set (simply Fs-closed) if  $Int(cl(A)) \leq A$. 
\item \cite{N} Generalized closed (Fg-closed) fuzzy set iff $cl(A) \leq U$  whenever $A \leq U$ and $U \in \tau$.
\item \cite{w8} Semi-generalized closed fuzzy set  (Fsg-closed) if $scl(A) \leq U $ whenever $A \leq U$ and $U$ is a   semi open fuzzy set.
\item \cite{w9} $\alpha-$generalized closed (F$\alpha$g-closed) fuzzy set if $\alpha cl(A) \leq U$  whenever $A \leq U$ and $U \in \tau$.
\item \cite{w9} Generalized $\alpha-$closed fuzzy set (Fg$\alpha$-closed) if $\alpha cl(A) \leq U$ whenever $A \leq U$ and $U$ is F$\alpha-$open.
\item \cite{w11} Semi-pre-generalized fuzzy set (Fspg-closed) if $spcl(A) \leq U$ whenever $A \leq U$ and $U$ is fuzzy semi-pre-open. 
\item \cite{w5} Pre-generalized fuzzy set (Fpg-closed) if $pcl(A) \leq U $ whenever $A \leq U$ and $U$ is fuzzy pre-open. 
\item \cite{K} Generalized-semi closed (Fgs-closed) fuzzy set if $scl(A) \leq U $ whenever $A \leq U$ and $U \in \tau$. 
\item Generalized-pre closed (Fgp-closed) fuzzy set if $pcl(A) \leq U $ whenever $A \leq U$ and $U \in \tau$. 
\item Generalized-semi-pre closed (Fgsp-closed) fuzzy set if $spcl(A) \leq U $ whenever $A \leq U$ and $U \in \tau$.
\item  \cite{XX} g*s-closed fuzzy set iff $scl(A) \leq U $ whenever $A \leq U$ and $U$ is Fgs-open fuzzy set. 
\end{enumerate}
\end{defn} 

The complements of the above closed sets are known as the respective open sets. 

\begin{prop}
 In an FTS $(X, \tau)$ the following results hold:
\begin{enumerate}
 \item \cite{N} Every fuzzy closed set is Fg-closed.
\item \cite{w8} Every Fs-closed set is Fsg-closed.
\item \cite{w9} Every F$\alpha$-closed set is F$\alpha$g-closed.
\item \cite{w5} Every Fp-closed set is Fpg-closed.
\item \cite{w11} Every Fsp-closed set is Fspg-closed.
\item \cite{XX} Every Fs-closed set is Fg*s-closed.
\item \cite{XX} Every g*s-closed fuzzy set is Fsg-closed.
\end{enumerate}
\end{prop}

The converse of the each of the above results are not necessarily true.

\begin{prop}(\cite{M}, \cite{w4}, \cite{w12})
 In an FTS $(X, \tau)$ the following results hold for different types of open sets:
\begin{enumerate}
\item Every fuzzy open set is F$\alpha$g-open.
\item Every F$\alpha$-open set is both fuzzy semiopen and fuzzy preopen.
\item Every fuzzy semiopen set is fuzzy semi-preopen
\end{enumerate}
\end{prop}

The following lemma is a direct consequence of the definitions.

\begin{lem}
 Let $A$ be a fuzzy subset in an FTS $(X, \tau)$. Then each of the following follows directly from the definitions:
\begin{enumerate}
 \item $spcl(A) \leq scl(A) \leq \alpha cl(A) \leq cl(A)$,
\item $spcl(A) \leq pcl(A) \leq \alpha cl(A)$
\end{enumerate}

\end{lem}

\section{Fuzzy weakly-closed sets} 

Weakly closed set in topological space was introduced by Sundaram {\it {et al}}. in \cite{PA}. In this section, we introduce fuzzy weakly-closed set in an FTS and study some general properties of such sets. 

\begin{defn} Let $(X, \tau)$ be an FTS and $A$ be a fuzzy subset of $X$. Then $A$ is said to be fuzzy weakly-closed
if $cl(A)\leq U$, whenever $A \leq U,~ U$ is a fuzzy semi-open set in $X$. Complement of a fuzzy weakly-closed set is fuzzy weakly-open set.
\end{defn}

\begin{exm}
Consider the FTS $(X, \tau)$, where $X =\{a, b, c\}, ~\tau = \{0_X, 1_X, B\}$ and\\

\begin{math}
 B(x) =\begin{cases}1  & if ~x = b; \\ 0  & otherwise. \end{cases}\\ 
\end{math}

Then  
\begin{math}
 F(x) =\begin{cases}1  & if ~x = a, c; \\ 0  & otherwise, \end{cases} \\
\end{math}

is fuzzy weakly-closed, whereas \\

\begin{math}
E(x) =\begin{cases}1  & if ~x = a, b; \\ 0  & otherwise, \end{cases}\\ 
\end{math}

is not fuzzy weakly-closed.
\end{exm}

\begin{thm} 
Union of two fuzzy weakly-closed set is fuzzy weakly-closed.
\end{thm}

\begin{proof}
Let $A$ and $B$ be two fuzzy weakly-closed sets in $X$. Let $U$ be a fuzzy semi-open set in $X$ such that $A \vee B \leq U \Rightarrow A \leq U$ and $B \leq U.$ Then by definition of fuzzy weakly closed set, we have $cl(A) \leq U$ and $cl(B) \leq U \Rightarrow cl(A \vee B) \leq U \Rightarrow A \vee B$ is fuzzy weakly-closed.
\end{proof}

\begin{rem} Intersection of two fuzzy weakly-closed sets is not necessarily fuzzy weakly-closed.
\end{rem}

\begin{exm}
Consider the fuzzy sets $A, I, N, F, E$ defined on $X =\{a, b, c, d\}$ as follows:\\

\begin{math}
A(x) =\begin{cases}1  & if ~x = a; \\ 0  & otherwise, \end{cases}\\
\end{math}

\begin{math}
I(x) =\begin{cases}1  & if ~x = b, d; \\ 0  & otherwise, \end{cases} \\
\end{math}

\begin{math}
N(x) =\begin{cases}1  & if ~x = a, b, d; \\ 0  & otherwise, \end{cases}\\ 
\end{math}

\begin{math}
F(x) =\begin{cases}1  & if ~x = a, c; \\ 0  & otherwise, \end{cases}\\ 
\end{math}
 
and 
\begin{math}
E(x) =\begin{cases}1  & if ~x = a, b; \\ 0  & otherwise \end{cases} \\
\end{math}

Then $\tau = \{0_X, A, I, N, 1_X\}$ is a fuzzy topology on $X$. Here $F$ and $E$ are fuzzy weakly-closed sets but $F \wedge E = A $ is not a fuzzy weakly closed set.
\end{exm}

\begin{rem}
Arbitrary union of fuzzy weakly-closed sets is not necessarily fuzzy weakly-closed.
\end{rem}

\begin{thm} 
Intersection of two fuzzy weakly-open sets is fuzzy weakly-open.
\end{thm}

\begin{proof}
Let $A$ and $B$ be two fuzzy weakly-open sets $\Rightarrow A^c$ and $B^c$ are fuzzy weakly-closed sets
 $\Rightarrow A^c \vee  B^c $ is a fuzzy weakly-closed set $\Rightarrow (A \wedge B)^c $ is a fuzzy weakly-closed set 
$\Rightarrow A \wedge B $ is a fuzzy weakly-open set.
\end{proof}

\begin{rem}
Union of two fuzzy weakly-open sets need not be fuzzy weakly-open.
\end{rem}

\begin{rem}
Arbitrary intersection of fuzzy weakly-open sets need not be fuzzy weakly-open.
\end{rem}

\begin{thm}
A fuzzy set $A$ in an FTS $(X, \tau)$ is fuzzy weakly open iff $U \leq int(A)$ whenever $U \leq A$ and $U$ is fs-closed.
\end{thm}

\begin{proof}
Let $A$ be fuzzy weakly open  set such that $U \leq A$ and $U$ is fs-closed $\Rightarrow A^c \leq U^c$, where $U^c$ is fs-open. Now $A^c$ is fuzzy weakly closed, so $cl(A^c) \leq U^c \Rightarrow U \leq (cl(A^c))^c = int(A)$.

Conversely, assume that $U \leq int(A)$ whenever $U \leq A$ where $U$ is fs-closed. Let $A^c \leq V$ where $V$ is fs-open $\Rightarrow V^c \leq A \Rightarrow V^c \leq int(A)$ (by hypothesis)$\Rightarrow (int(A))^c  \leq V \Rightarrow cl(A^c) \leq V \Rightarrow A^c$ is fuzzy weakly closed $\Rightarrow A$ is fuzzy weakly open.
\end{proof}

\begin{thm}
If a fuzzy set $A$ in an FTS $(X, \tau)$ is both fuzzy open and fuzzy generalized closed, then it is fuzzy weakly closed.
\end{thm}

\begin{proof}
Let $A$ be an fuzzy open set which is also f$g$-closed $\Rightarrow cl(A) \leq A$ as $A$ is a fuzzy open set containing itself. We know $A \leq cl(A)$ for any fuzzy set in $X$. Hence $A= cl(A) \Rightarrow A$ is fuzzy closed set 
 $\Rightarrow A$ is fuzzy weakly-closed set.
\end{proof}

\begin{rem}
If $A$ and $B$ are two fuzzy sets such that $A \leq B \leq cl(A)$, then $cl(A)=cl(B)$.
\end{rem}

\begin{thm}
Let $A$ be a fuzzy weakly-closed subset of an FTS $(X, \tau)$ and suppose the fuzzy subset $B$ is such that $A \leq B \leq cl(A)$, then $B$ is also fuzzy weakly-closed.
\end{thm}

\begin{proof}
Let $A$ be a fuzzy weakly-closed set and $U$ be an fuzzy semi-open set such that $B \leq U \Rightarrow A \leq U$. Then by definition of fuzzy weakly-closed set $cl(A) \leq U$. But $cl(B) = cl(A) \Rightarrow cl(B) \leq U \Rightarrow B$ is fuzzy weakly-closed set.
\end{proof}

\begin{defn}
Let $A$ be a fuzzy set in an FTS $X$. Then fuzzy weakly-closure and weakly-interior of $A$(denoted by $wcl(A)$
 and $wint(A)$ respectively) are defined as follows:
 
$wcl(A)= \bigwedge \{B ~|~ B$ is fuzzy weakly-closed set and $A\leq B\}$.

$wint(A)= \bigvee \{C ~|~ C$ is fuzzy weakly-open set and $C\leq A\}$.
\end{defn}

If $A$ is a fuzzy weakly-closed set then $wcl(A) = A$ but the converse is not necessarily true since intersection of any two fuzzy weakly-closed sets may not be fuzzy weakly-closed. 

Following results are obtained from the definitions:
\begin{center}
$wcl(U^c)= (wint(U))^c$ and $wint( U^c)= (wcl(U))^c$.
\end{center}

\section{Comparative study}

In this section, fuzzy weakly closed set is compared with different fuzzy closed sets. Below is a diagram describing inter relationship of fuzzy weakly closed sets and other types of fuzzy closed sets. As implications are direct outcomes of the definitions, we are omitting the proofs for the same. However, counter examples are given wherever necessary.

\begin{figure}[h]
\begin{center}
{\includegraphics[width=12cm, height=8cm]{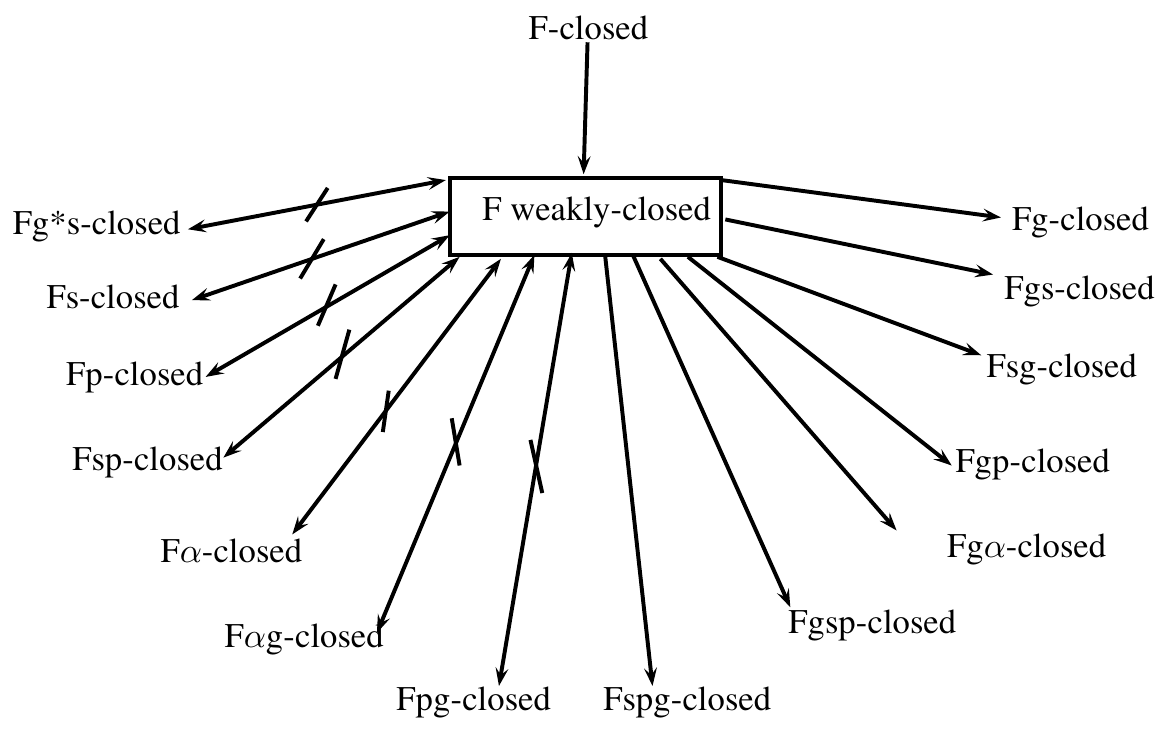}}
\caption{Comparison of fuzzy weakly-closed set with different fuzzy closed sets}
\end{center}
\end{figure}

As shown in the figure the following are some examples that the converses of some implications are not necessarily true and some of the classes of closed sets are independent of the class of fuzzy weakly closed sets.

\begin{exm}
Taking the same FTS as defined in Example 2.5, we have $N$ is a fuzzy weakly-closed set but not a fuzzy closed set.  
\end{exm}

\begin{exm}
For the FTS in Example 2.5,\\ 

\begin{math}
I(x) =\begin{cases}1  & if ~x = b, d; \\ 0  & otherwise \end{cases} \\
\end{math}

is  f$gs$-closed and f$sg$-closed set but not fuzzy weakly-closed set.
\end{exm}

\begin{exm}
For the FTS in Example 2.5,\\

\begin{math}
G(x) =\begin{cases}1  & if ~x = a, d; \\ 0  & otherwise \end{cases}\\ 
\end{math}

is a f$g$-closed set but not fuzzy weakly-closed set. 
\end{exm}

\begin{exm}
But taking the FTS in Example 2.5, $A$ is a f$g \alpha$-closed set but not fuzzy weakly-closed set. 
\end{exm}

\begin{exm}
Again for the FTS in Example 2.5, $A$ is a f$spg$-closed set but not fuzzy weakly-closed set.
\end{exm}

\begin{exm}
Let $X =\{a, b, c\}$ and $H$ be a fuzzy set on it defined by\\

\begin{math}
E(x) =\begin{cases}1  & if ~x = a, b; \\ 0  & otherwise. \end{cases} \\
\end{math}

Then  $\tau = \{0_X, 1_X, E\}$ is a fuzzy topology on $X$. Here $E$ is a f$gsp$-closed but not fuzzy weakly-closed.\\
\end{exm}

\begin{exm}
For Example 2.5,\\

\begin{math}
C(x) =\begin{cases}1  & if ~x = c; \\ 0  & otherwise \end{cases} 
\end{math}
is a f$gp$-closed set but not a fuzzy weakly-closed set.
\end{exm}

\begin{exm}
Taking the same example as Example 2.5, $A$ is a f$s$-closed, f$\alpha$-closed and f$g^*s$-closed, $N$ is a f$sp$-closed, $C$ is  f$p$-closed and f$gp$-closed but are not fuzzy weakly-closed; on the other hand $E$ is fuzzy weakly-closed set but not f$s$-closed, f$\alpha$-closed, f$p$-closed set and f$pg$-closed set and\\

\begin{math}
J(x) =\begin{cases}1  & if ~x = c, d; \\ 0  & otherwise \end{cases} 
\end{math} 
is fuzzy weakly-closed but not f$sp$-closed as well as f$g^*s$-closed.
\end{exm}

\section{Compactness}

A topological space is said to be compact if every open cover of it has a finite subcover, i.e., the property of a topological space being compact depends on open sets. As open set is generalized to weakly open set, the concept of compactness can also be generalized in the similar manner. In this section, we introduce fuzzy weakly-compactness and fuzzy weakly-closed space, which are then characterized using fuzzy filterbases.\\
  
The notion of compactness for FTS was introduced by C.L. Chang \cite{CA}. But Lowen \cite{RL} differed from this and subsequently definition for compactness by Chang is popularly known as quasi compactness.\\

We present below a result showing relationship of f$g^*s$-closed sets and the so called quasicompactness. 

\begin{thm}
 Let $(X, \tau)$ be a fuzzy quasi compact topological space. If $A$ is a fuzzy weakly-closed subset of $X$, then $A$ is also quasi compact.
\end{thm}

\begin{proof}
Let $\{U_{\lambda} ~|~\lambda \in \Lambda\}$ be a fuzzy open cover of $A$, i.e., $A \leq \underset{\lambda \in \Lambda}{\bigvee} U_\lambda$. But $\underset{\lambda \in \Lambda}{\bigvee} U_\lambda$ is f$s$-open set containing $A$, which is fuzzy weakly-closed $\Rightarrow cl(A) \leq \underset{\lambda \in \Lambda}{\bigvee} U_{\lambda}$. We have $cl(A)$ is a closed fuzzy set and hence is quasi compact
$\Rightarrow \exists$ a finite subfamily $\{U_{\lambda_i} ~|~i = 1, 2, . . . ,n\}$ such that $cl(A) \leq \underset{i=1}{\overset{n}{\bigvee}} U_{\lambda_i} \Rightarrow A \leq \underset{i=1}{\overset{n}{\bigvee}} U_{\lambda_i} \Rightarrow A$ is fuzzy quasi compact. 
\end{proof}

\begin{defn}
An FTS $(X, \tau)$ is said to be fuzzy weakly-compact iff for every family $\{G_{\lambda} ~|~\lambda \in \Lambda\}$ of fuzzy weakly-open sets satisfying $\underset{\lambda \in \Lambda}{\bigvee} G_\lambda =1_X$, there is a finite subfamily $\{G_{\lambda_i} ~|~i = 1, 2, . . . ,n\}$ such that $\underset{i=1}{\overset{n}{\bigvee}} G_{\lambda_i} =1_X$.
\end{defn}

\begin{exm}
Let $X=[0, 1]$ and $A_n:X \rightarrow [0, 1]$ be defined by \\

\begin{math}
A_n(x) =\begin{cases}1  & if ~n~is~odd; \\ 1-\frac{1}{n}  & if ~n~is~even. \end{cases}
\end{math}\\

Then $\tau=\{0_X, 1_X, A_n\}$, where $n \in \mathbb{N}$ is a fuzzy topology on $X$. If $\mu$ is a weakly open cover of $X$, then $1_X$ must be a member of $\mu$, otherwise there does not exist a cover of $X$. Hence $X$ is fuzzy weakly-compact.
\end{exm}

\begin{defn}
A fuzzy set U in an FTS $X$ is said to be fuzzy weakly-compact relative to $X$ iff for every family $\{G_{\lambda} ~|~\lambda \in \Lambda\}$ of fuzzy weakly-open sets satisfying $\underset{\lambda \in \Lambda}{\bigvee} G_\lambda (x) \geq U(x),~\forall x\in Supp(U)$, there is a finite subfamily $\{G_{\lambda_i} ~|~i = 1, 2, . . . ,n\}$  such that $(\underset{i=1}{\overset{n}{\bigvee}} G_{\lambda_i})(x) \geq U(x),$ for every $x \in Supp(U)$.
\end{defn}

\begin{thm}
An FTS $X$ is fuzzy weakly-compact iff $X$ does not contain a fuzzy filterbase of fuzzy weakly-closed sets such that the corresponding collection of fuzzy weakly-open sets forms a cover of $X$.
\end{thm}

\begin{proof}
Equivalently we show, an FTS $X$ is not fuzzy weakly-compact iff $X$ contains atleast one fuzzy filterbase of fuzzy weakly-closed sets such that the corresponding collection of fuzzy weakly-open sets forms a cover of $X$.

Let $X$ be an FTS which is not fuzzy weakly-compact $\Rightarrow~\exists$ a fuzzy weakly-open cover of $X$, say $\{A_\lambda~|~ \lambda \in \Lambda\}$ without a finite subcover. i.e., every finite subcollection $\{A_{\lambda_i}~|~ i = 1, 2, . . . , n\}$ is such that $\underset{i=1}{\overset{n}{\bigvee}} A_{\lambda_i} \neq 1_X \Rightarrow \underset{i=1}{\overset{n}{\bigwedge}} A_{\lambda_i}^c \neq 0_X \Rightarrow \{A_\lambda^c~|~\lambda \in \Lambda\}$ forms a fuzzy filterbase of fuzzy weakly-closed sets.

Conversely, assume $X$ contains atleast one fuzzy filterbase of fuzzy weakly-closed sets in $X$. If possible, let $X$ be fuzzy weakly-compact. So, every fuzzy weakly-open cover of $X$, say $\{A_\lambda~|~ \lambda \in \Lambda\}$, has a finite subcover, $\{A_{\lambda_i}~|~ i = 1, 2, . . . , n\}$, i.e., $\underset{i=1}{\overset{n}{\bigvee}} A_{\lambda_i} = 1_X \Rightarrow \underset{i=1}{\overset{n}{\bigwedge}} A_{\lambda_i}^c = 0_X \Rightarrow$ for every collection of fuzzy weakly-closed sets in $X$, there is atleast one subcollection with empty intersection. Therefore, a collection of fuzzy weakly-closed sets can not form a fuzzy filterbase.
\end{proof}

\begin{thm}
An FTS $X$ is fuzzy weakly-compact if for every fuzzy filterbase $\Gamma $ in $X$, $\underset{A \in \Gamma}{\wedge} wcl(A) \neq 0_X$.
\end{thm}

\begin{proof}
If possible, suppose $\{A_\lambda~|~ \lambda \in \Lambda\}$ is a fuzzy weakly-open cover of $X$ which does not have a finite subcover. Then for every finite subcollection $\{A_{\lambda_i}~|~ i = 1, 2, . . . , n\}$, there exists $x \in X$ such that $A_{\lambda_i}(x) < 1$ for each $i = 1, 2, . . . , n.$ Then 
${A_{\lambda_i}}^c > 0_X \Rightarrow \overset{n}{\underset{i=1}{\bigwedge}} {A_{\lambda_i}}^c \neq 0_X \Rightarrow \{{A_\lambda}^c~|~ \lambda \in \Lambda\}$ 
forms a fuzzy filterbase of fuzzy weakly-closed set in $X$. Since $\{A_\lambda~|~ \lambda \in \Lambda\}$ is fuzzy weakly-open cover of $X$, $(\underset{\lambda \in \Lambda}{\vee}A_\lambda)(x)=1$ 
for every $x \in X$ and hence $(\underset{\lambda \in \Lambda}{\bigwedge} {A_\lambda}^c)(x)=0 \Rightarrow (\underset{\lambda \in \Lambda}{\bigwedge} wcl {A_\lambda}^c)(x)=0$
, a contradiction. Hence $X$ is fuzzy weakly-compact.
\end{proof}

\begin{thm}
If an FTS $X$ is fuzzy weakly-compact then for every fuzzy filterbase $\Gamma $ of fuzzy weakly-closed sets in $X, \underset{A \in \Gamma}{\wedge} cl(A) \neq 0_X$.
\end{thm}

\begin{proof}
Let $X$ be a weakly-compact FTS. If possible, assume that, there exists a filterbase $\Gamma$ of weakly-closed fuzzy sets in $X$ such that 
$\underset{A \in \Gamma}{\bigwedge}cl (A)=0_X \Rightarrow \underset{A \in \Gamma}{\bigvee}{(clA)}^c=1_X \Rightarrow \{{(clA)}^c~|~A \in \Gamma \}$ is a fuzzy weakly-open cover of $X$, Then by definition of fuzzy weakly-compactness, $\exists$ a finite subcollection $\{{(clA_i)}^c~|~A \in \Gamma, i= 1, 2, . . . ,n \}$ 
such that $\overset{n}{\underset{i=1}{\bigvee}}  {(clA_i)}^c=1_X \Rightarrow (\overset{n}{\underset{i=1}{\bigvee}}{A_i}^c)=1_X  \Rightarrow (\overset{n}{\underset{i=1}{\bigwedge}} A_i)=0_X$, a contradiction. Hence $\underset{A \in \Gamma}{\wedge} cl(A) \neq 0_X$. 
\end{proof}

\begin{thm}
A fuzzy subset $U$ in an FTS $X$ is fuzzy weakly-compact relative to $X$ if for every fuzzy filterbase $\Gamma$ of fuzzy weakly-closed sets in $X$,  every finite subcollection of $\Gamma$ is quasi coincident with $U$ and  $(\underset{G \in \Gamma}{\bigwedge}wclG)~\wedge U \neq 0_X$. 
\end{thm}

\begin{proof}
Assume $U$ is not fuzzy weakly-compact relative to $X$, then $\exists$ a fuzzy weakly-open covering $\{A_\lambda~|~ \lambda \in \Lambda \}$ of $U$ without any finite subcover. \\This implies
$(\underset{i=1}{\overset{n}{\bigvee}}A_{\lambda_i})(x) < U(x)$ for some $x \in Supp(U),$ for every finite subfamily  of $\{A_{\lambda_i}~|~ i=1, 2, . . . ,n\}$. So, $(\underset{i=1}{\overset{n}{\bigwedge}}A_{\lambda_i}^c)(x)  > 0 \Rightarrow \Gamma = \{ A_\lambda^c~|~ \lambda \in \Lambda\}$ forms a filterbase of weakly-closed fuzzy sets in $X$ and $(\underset{i=1}{\overset{n}{\bigvee}}A_{\lambda_i}^c)~q~ U$. Next by hypothesis $(\underset{A_\lambda^c \in \Gamma}{\bigwedge}wcl(A_\lambda^c))~\wedge U \neq 0_X \Rightarrow (\underset{A_\lambda^c \in \Gamma}{\bigwedge}A_\lambda^c)~\wedge U \neq 0_X.$ Then for some $x \in Supp(U), (\underset{A_\lambda^c \in \Gamma}{\bigwedge}A_\lambda^c)(x) > 0 \Rightarrow (\underset{A_\lambda^c \in \Gamma}{\bigvee}A_\lambda)(x) < 1,$ a contradiction. Hence $U$ is fuzzy weakly-compact relative to $X$.
\end{proof}

\begin{thm}
If  $U$ is a fuzzy weakly-compact set relative to an FTS $X$ then for every fuzzy filterbase $\Gamma$ of fuzzy weakly-closed sets in $X$, every finite subcollection of $\Gamma$ is quasi coincident with $U$ and $(\underset{G \in \Gamma}{\bigwedge}clG)~\wedge U \neq 0_X$. 
\end{thm}

\begin{proof}
Let $U$ be fuzzy weakly-compact relative to $X$ and $\exists$ a filterbase $\Gamma$ of fuzzy weakly-closed sets in $X$ such that every finite subcollection of $\Gamma$ is quasi coincident with $U$ and $(\underset{G \in \Gamma}{\bigwedge}cl G)~\bigwedge U = 0_X \Rightarrow$ for every $x \in Supp(U),~ (\underset{G \in \Gamma}{\bigwedge}cl G)(x) = 0 \Rightarrow (\underset{G \in \Gamma}{\bigvee}(clG)^c)(x) = 1 \Rightarrow \{(clG)^c~|~ G \in \Gamma\}$ is a fuzzy weakly open covering of $U$. But $U$ is fuzzy weakly-compact relative to $X$, so $\exists$ a finite subfamily $\{(clG_i)^c~|~ i= 1, 2, . . . ,n\}$ such that for all $ x \in Supp(U),~(\underset{i=1}{\overset{n}{\bigvee}}(clG_i)^c)(x) \geq U(x) \Rightarrow (\underset{i=1}{\overset{n}{\bigwedge}}(clG_i))(x) \leq U^c(x) \Rightarrow (\underset{i=1}{\overset{n}{\bigwedge}}clG_i) ~{\overline q}~ U,$ a contradiction. Hence the result. 
\end{proof}

\begin{defn}
An FTS $X$ is said to be fuzzy weakly-closed iff for every family $\{G_\lambda~|~ \lambda \in \Lambda\}$ of fuzzy weakly-open sets with $\underset{\lambda \in \Lambda}{\bigvee} G_\lambda(x)=1,$ there exists a finite subfamily $\{G_{\lambda_i}~|~ i = 1, 2, . . . ,n\}$ such that $(\underset{i=1}{\overset{n}{\bigvee}}wcl G_{\lambda_i})(x)=1$ for every $x \in X$.
\end{defn}

\begin{defn}
A fuzzy set $U$ in an FTS $X$ is said to be fuzzy weakly-closed relative to $X$ iff for every family $\{G_\lambda~|~ \lambda \in \Lambda\}$ of fuzzy weakly-open sets with $\underset{\lambda \in \Lambda}{\bigvee} G_\lambda(x)=U(x)$, for all $x \in Supp(U)$, there exists a finite subfamily 
$\{G_{\lambda_i}~|~ i = 1, 2, . . . ,n\}$ such that $(\underset{i=1}{\overset{n}{\bigvee}}wcl G_{\lambda_i})(x)=U(x)$, for every $x \in Supp(U)$.
\end{defn}

\begin{rem} Every fuzzy weakly-compact space is fuzzy weakly-closed but not conversely.
\end{rem}

\begin{thm}
An FTS $X$ is fuzzy weakly-closed if for every fuzzy filterbase $\Gamma$ in $X, (\underset{A \in \Gamma}{\bigwedge}wclA)\neq 0_X$.
\end{thm}

\begin{proof}
Let every fuzzy filterbase $\Gamma$ in $X$ be such that $(\underset{A \in \Gamma}{\bigwedge}wclA)\neq 0_X$. If possible, assume $\{G_\lambda~|~ \lambda \in \Lambda\}$ be a fuzzy weakly-open cover of $X$ and let for every finite subfamily $\{G_{\lambda_i}~|~ i = 1, 2, . . . ,n\}, (\underset{i=1}{\overset{n}{\bigvee}}wclG_{\lambda_i})(x) < 1$ for some $x \in X \Rightarrow (\underset{i=1}{\overset{n}{\bigwedge}}wclG_{\lambda_i}^c)(x)> 0$, for some $x \in X \Rightarrow  \{{(wclG_\lambda)}^c~|~ \lambda \in \Lambda\}$ forms a fuzzy filterbase in $X$. We have $\underset{\lambda \in \Lambda}{\bigwedge} G_\lambda^c(x) = 0 \Rightarrow \underset{\lambda \in \Lambda}{\bigwedge}wcl ({wclG_\lambda)}^c(x) = 0$, a contradiction. Hence $X$ is fuzzy weakly-closed.
\end{proof}

\begin{thm}
 If an FTS $X$ is fuzzy weakly-closed then for every fuzzy weakly-open filterbase $\Gamma$ in $X, (\underset{G \in \Gamma}{\bigwedge}cl(G))\neq 0_X$.
\end{thm}

\begin{proof}
Let $X$ be a fuzzy weakly-closed space, assume there exists a fuzzy weakly-open filterbase $\Gamma$ in $X$ such that $\underset{G \in \Gamma}{\bigwedge}cl(G)=0_X \Rightarrow \underset{G \in \Gamma}{\bigvee}(clG)^c=1_X \Rightarrow  \{(clG_\lambda)^c~|~\lambda \in \Lambda, ~G_\lambda \in \Gamma\}$ is a fuzzy weakly-open covering of $X$. Then by the definition of fuzzy weakly-closed space, it has a finite subfamily $\{({clG_{\lambda_i})}^c~|~i = 1, 2, . . . , n\}$ such that $\underset{i =1}{\overset{n}{\bigvee}}wcl({clG_{\lambda_i})}^c=1_X 
\Rightarrow \underset{i =1}{\overset{n}{\bigwedge}}({wcl({clG_{\lambda_i})}^c)}^c=0_X \Rightarrow \underset{i =1}{\overset{n}{\bigwedge}} G_{\lambda_i} =0_X$, a contradiction since all $G$'s are members of a filterbase. Hence $\underset{G \in \Gamma}{\bigwedge}cl(G) \neq 0_X$. 
\end{proof}

\begin{thm}
If a fuzzy subset $U$ in an FTS $X$ is fuzzy weakly-closed relative to $X$ then for every fuzzy weakly-open filterbase $\{G_\lambda ~|~ \lambda \in \Lambda\}$ in $X$ with $(\underset{\lambda \in \Lambda}{\bigwedge} cl(G_\lambda))\wedge U \neq 0_X$, there exists a finite subfamily $\{G_{\lambda_i}~|~ i = 1, 2, . . ., n\}$ such that $(\underset{i =1}{\overset{n}{\bigwedge}}G_{\lambda_i})~q~U$. 
\end{thm}

\begin{proof}
Let $U$ be fuzzy weakly-closed relative to $X$ and $\{G_\lambda ~|~ \lambda \in \Lambda\}$ be a fuzzy weakly-open filterbase in $X$, such that for every finite subfamily $\{G_{\lambda_i}~|~ i = 1, 2, . . ., n\}$ of $\{G_\lambda ~|~ \lambda \in \Lambda\}$, we have $(\underset{i=1}{\overset{n}{\bigwedge}}G_{\lambda_i})~q~U$ but $(\underset{\lambda \in \Lambda}{\bigwedge} cl(G_\lambda))\wedge U = 0_X \Rightarrow (\underset{\lambda \in \Lambda}{\bigwedge}cl(G_\lambda))(x) = 0$ for every 
$x \in ~Supp(U)\Rightarrow \underset{\lambda \in \Lambda}{\bigvee} {(clG_\lambda)}^c(x) = 1$ for every $x \in ~Supp(U)
 \Rightarrow \{({clG_\lambda)}^c~|~ \lambda \in \Lambda\}$ forms a fuzzy weakly-open cover of $U$ and hence there is a
 finite subfamily $\{cl(G_{\lambda_i})^c ~|~ i = 1,2, . ., n\}$ such that $\underset{i =1}{\overset{n}{\bigvee}} wcl{(clG_{\lambda_i})}^c
\geq U \Rightarrow \underset{i =1}{\overset{n}{\bigwedge}}({wcl {(clG_{\lambda_i})}^c)}^c \leq U^c \Rightarrow \underset{i =1}{\overset{n}{\bigwedge}}G_{\lambda_i} 
\leq U^c \Rightarrow \underset{i =1}{\overset{n}{\bigwedge}}G_{\lambda_i}~\overline{q}~ U$, a contradiction.
\end{proof}

\begin{thm}
A fuzzy subset $U$ in an FTS $X$ is fuzzy weakly-closed relative to $X$ if for every fuzzy filterbase $\{G_\lambda ~|~ \lambda \in \Lambda\}$ in $X$ with $(\underset{\lambda \in \Lambda}{\bigwedge} wclG_\lambda)\wedge U \neq 0_X$, there exists a finite subfamily $\{G_{\lambda_i}~|~ i = 1, 2, . . ., n\}$ such that $(\underset{i =1}{\overset{n}{\bigwedge}}G_{\lambda_i})~q~U$. 
\end{thm}

\begin{proof}
Suppose $U$ satisfies the hypothesis and if possible, let $U$ be not a fuzzy weakly-closed set relative to $X$, then there exists a collection  $\{G_\lambda ~|~ \lambda \in \Lambda\}$ of fuzzy weakly-open sets that covers $U$ such that for every finite subfamily $\{G_{\lambda_i}~|~ i = 1, 2, . . ., n\}$ we have $(\underset{i=1}{\overset{n}{\bigvee}} wclG_{\lambda_i})(x) \leq U(x)$ for some $x \in Supp(U)$ and hence $\underset{i=1}{\overset{n}{\bigwedge}}(wclG_{\lambda_i})^c(x) \geq {U^c(x)} \geq 0$ for some $x \in Supp(U)\Rightarrow \{{(wclG_{\lambda})}^c~|~ \lambda \in \Lambda\}$ forms a fuzzy filterbase in $X$. 

Now we claim that, $\underset{i =1}{\overset{n}{\bigwedge}} {(wclG_{\lambda_i})}^c~ q ~U $ for every finite subfamily 
of $\{G_\lambda ~|~ \lambda \in \Lambda\}$, for otherwise $\underset{i =1}{\overset{n}{\bigwedge}} {(wclG_{\lambda_i})}^c ~ \overline{q}~U \Rightarrow U \leq \underset{i =1}
{\overset{n}{\bigvee}}wclG_{\lambda_i}$, a contradiction.

So, $ \underset{i =1}{\overset{n}{\bigwedge}} 
wcl{(wclG_{\lambda_i})}^c~ \bigwedge~U \neq 0_X \Rightarrow \exists $ at least one $x \in Supp(U)$ such that $\underset{i=1}{\overset{n}{\bigwedge}}wcl{(wclG_{\lambda_i})}^c~ > 0_X \Rightarrow \underset{i=1}{\overset{n}{\bigvee}} {(wcl{(wclG_{\lambda_i})}^c)}^c  < 1_X$ and hence $\underset{i=1}{\overset{n}{\bigvee}}G_{\lambda_i} < 1_X$, a contradiction.
\end{proof}

\section{Conclusion}

In this paper, we have introduced fuzzy weakly closed set in an FTS and studied some of its set theoretic properties. We have also examined inter relationship of fuzzy weakly-closed sets and other generalizations of closed sets in FTS. Further, fuzzy compactness and fuzzy closed space are 
 studied through fuzzy weakly closed sets. \\

\end{document}